\documentclass[11pt,reqno]{amsart}
\usepackage{amsmath,amssymb,amsfonts,amsthm,enumerate,natbib,color,ifthen,hyperref}
\usepackage{xargs}
\usepackage{accents}
\usepackage[textwidth=4cm, textsize=footnotesize]{todonotes}
\usepackage[latin1]{inputenc}

\usepackage{float}
\usepackage[caption = false]{subfig}
\usepackage{graphicx}

\usepackage{aliascnt,bbm}
\def\rset{\mathbb R}
\def\zset{\mathbb Z}

\def\eqsp{\;}

 \newcommand{\pscal}[2]{\left\langle#1,#2\right\rangle}

\newcommand{\eqdef}{\ensuremath{\stackrel{\mathrm{def}}{=}}}

\def\Yset{\mathsf{Y}} 
\def\e{\mathcal{E}}

\def\N{\mathcal{N}}
\def\M{\mathcal{M}}

\def\A{\mathcal{A}}

\newcommandx\sequence[3][2=t,3=\zset]
{\ifthenelse{\equal{#3}{}}{\ensuremath{\{ #1_{#2}\}}}{\ensuremath{\{ #1_{#2}, \eqsp #2 \in #3 \}}}}

\def\PP{\mathbb{P}} 
\newcommand{\CPP}[3][]
{\ifthenelse{\equal{#1}{}}{{\mathbb P}\left(\left. #2 \, \right| #3 \right)}{{\mathbb P}_{#1}\left(\left. #2 \, \right | #3 \right)}}
\def\PE{\mathbb{E}} 
\newcommand{\CPE}[3][]
{\ifthenelse{\equal{#1}{}}{{\mathbb E}\left[\left. #2 \, \right| #3 \right]}{{\mathbb E}_{#1}\left[\left. #2 \, \right | #3 \right]}}

\def\K{\vartheta}
\def\L{\mathcal{L}} 

\newcommand{\nnorm}[1]{\left\vert\!\left\vert\!\left\vert#1\right\vert\!\right\vert\!\right\vert}

\newcommand{\utilde}[1]{\underaccent{\tilde}{#1}}

\def\r{\textsf{r}}

\usepackage{bm}

\theoremstyle{plain}
\newtheorem{theorem}{Theorem}
\newtheorem{assumption}{H\hspace{-3pt}}

\newaliascnt{proposition}{theorem}

\aliascntresetthe{proposition}

\newaliascnt{lemma}{theorem}
\newtheorem{lemma}[lemma]{Lemma}
\aliascntresetthe{lemma}
\newaliascnt{corollary}{theorem}

\aliascntresetthe{corollary}

\theoremstyle{definition}
\newaliascnt{definition}{theorem}

\aliascntresetthe{definition}

\newaliascnt{remark}{theorem}
\newtheorem{remark}[remark]{Remark}
\aliascntresetthe{remark}

\newaliascnt{example}{theorem}
\newtheorem{example}[example]{Example}
\aliascntresetthe{example}

\def\rmd{\mathrm{d}}

\def\1{\mathbbm{1}}

\begin{document}

\title[A scalable quasi-Bayesian framework for Gaussian graphical models]{A scalable quasi-Bayesian framework for Gaussian graphical models}

\author{Yves F. Atchad\'e}  \thanks{ Y. F. Atchad\'e: University of Michigan, 1085 South University, Ann Arbor,
  48109, MI, United States. {\em E-mail address:} yvesa@umich.edu}



\maketitle

\begin{center} (Dec. 2015) \end{center}

\begin{abstract}
This paper deals with the Bayesian estimation of high dimensional Gaussian graphical models. We develop a quasi-Bayesian implementation of the neighborhood selection method of \cite{meinshausen06} for the estimation of Gaussian graphical models. The method produces a product-form quasi-posterior distribution that can be efficiently explored by parallel computing. We derive a non-asymptotic bound on the  contraction rate of the quasi-posterior distribution. The result shows that the proposed quasi-posterior distribution contracts towards the true precision matrix at a rate given by the worst contraction rate of the linear regressions that are involved in the neighborhood selection. We develop a Markov Chain Monte Carlo algorithm for approximate computations, following an approach from \cite{atchade:15a}. We illustrate the methodology with a simulation study. The results show that the proposed method can fit Gaussian graphical models at a scale unmatched by other Bayesian methods for graphical models.
\end{abstract}

\setcounter{secnumdepth}{3}

\section{Introduction}\label{sec:intro}
We consider the problem of fitting large Gaussian graphical models with diverging number of parameters from limited data. This amount to estimating a sparse precision matrix $\vartheta\in \M_p^+$ from $p$-dimensional Gaussian observations $y^{(i)}\in\rset^p$, $i=1,\ldots,n$, where $\M_p^+$ denotes the cone of $\rset^{p\times p}$ of symmetric positive definite matrices. The frequentist approach to this problem has generated an impressive literature over the last decade or so (see for instance \cite{buhlGeer11,hastie:etal:15} and the reference therein). 

There is currently an interest, particularly in biomedical research, for statistical methodologies that can allow practitioners to incorporate external information in fitting such graphical models (\cite{mukherjee:08,peterson:etal:14}). 
This problem naturally calls for a Bayesian formulation (\cite{dobra:11a,dobra:11b,wang:li:12,khondkeretal13,peterson:etal:14}). However, most existing Bayesian methods for fitting graphical models do not scale well with $p$, the number of nodes in the graph. The main difficulty is computational, and hinges on the ability to handle interesting prior distributions on $\M_p^+$. The most commonly used class of priors distributions for Gaussian graphical models is the class of G-Wishart distributions (\cite{atay:massam:05}). However G-Wishart distributions have intractable normalizing constants, and become impractical for inferring large graphical models, due to the cost of approximating the normalizing constants (\cite{dobra:11a,dobra:11b,wang:li:12}).  Following the development of the Bayesian lasso of \cite{casella:park08} and other Bayesian shrinkage priors for linear regressions (\cite{carvalhoetal10}), several authors have proposed prior distributions on $\M_p^+$ obtained by putting conditionally independent shrinkage priors on the entries of the matrix, subject to a positive definiteness constraint (\cite{wang12,khondkeretal13}). The main drawback of this approach is that these prior distributions are constructed explicitly so as to cancel the intractable normalizing constants, raising the issue of the impact of such prior distribution trick on the inference. Furthermore, dealing with the positive definiteness constraint in the posterior distribution requires careful MCMC design, and becomes a limiting factor for large $p$. So it appears that most existing Bayesian methods for high-dimensional graphical models do not scale well, and can fit only small to moderately large models (upto $p=200$).


Building on some recent works \cite{atchade:15a,atchade:15b}, we develop a quasi-Bayesian approach for fitting large Gaussian graphical models. Our general approach to the problem consists in working with a ``larger" pseudo-model $\{\check f_\theta,\;\theta\in\check\Theta\}$, where $\M_p^+\subseteq\check\Theta$. By pseudo-model we mean that the function $z\mapsto \check f_\theta(z)$ is typically not a density on $\rset^p$, but $\check f_\theta$ is chosen such that the function $\theta\mapsto \sum_{i=1}^n \log \check f_\theta(y^{(i)})$ is a good candidate for M-estimation of $\vartheta$. The enlargement of the model space from $\M_p^+$ to $\check\Theta$ allows us to relax the positive definiteness constraint. With a prior distribution $\Pi$ on $\check\Theta$, we obtain a quasi-posterior distribution $\check\Pi_{n,p}$ (not a proper posterior distribution), since $\check f_\theta$ is not a proper likelihood function.  In the specific case of Gaussian graphical models, we propose to take $\check\Theta$ as the space of matrices with positive diagonals, and to take $z\mapsto \check f_\theta(z)$ as the pseudo-model underpinning the neighborhood selection method of \cite{meinshausen06}. This choice  gives a quasi-posterior distribution $\check\Pi_{n,p}$ that factorizes, and leads to a drastic improvement in the computing time needed for MCMC computation when a parallel computing architecture is used. We illustrate the method in Section \ref{sec:example} using simulated data where the number of nodes in the graph is $p\in\{100,500,1000\}$.

The idea of replacing the likelihood function by a pseudo-likelihood (or quasi-likelihood) function in a Bayesian inference is not new and has been developed in other contexts, such as in Bayesian semi-parametric inference (\cite{kato:13,li:jiang:14}, and the references therein), and in approximate Bayesian computation (\cite{fearnheadetal10}). A general analysis of the contraction properties of these distributions for high-dimensional problems can be found in \cite{atchade:15b}.

We study the contraction properties of the quasi-posterior distribution $\check\Pi_{n,p}$ as $n,p\to\infty$. Under the assumption that there exists a true precision matrix, and some additional assumption on the prior distribution, we show that when the true precision matrix is well conditioned, then $\check\Pi_{n,p}$ contracts\footnote{The contraction rate is measured in the $L_{\infty,2}$ matrix norm, defined as the largest $L_2$ column norms} at the rate $\sqrt{\frac{\bar s \log(p)}{n}}$ (see Theorem \ref{thm2} for a precise statement), where $\bar s$ can be viewed as an upper-bound on the largest degree in the un-directed graph defined by the true precision matrix. This convergence rate corresponds to the worst convergence rate that we get from the Bayesian analysis of the linear regressions involved in the neighborhood selection. The condition on the sample size $n$ for the results mentioned above to hold is  $n = O\left(\bar s\log(p)\right)$, which shows that the quasi-posterior distribution can concentrate around the true value, even in cases where $p$ exceeds $n$. 

The rest of the paper is organized as follows. Section \ref{sec:qbgm} provides a general discussion of quasi-models and quasi-Bayesian inference. The section ends with the introduction of the proposed quasi-Bayesian distribution, based on the neighborhood selection method of \cite{meinshausen06}. We specialized the discussion to Gaussian graphical models in Section \ref{sec:ggm}. The theoretical analysis focuses on the Gaussian case, and is presented in Section \ref{sec:ggm}, but the proofs are postponed to Section \ref{sec:proof}. The simulation study is presented in Section \ref{sec:example}. We end the paper with some concluding thoughts in Section \ref{sec:conclusion}. A \texttt{MATLAB} implementation of the method is available from the author's website.

\section{Quasi-Bayesian inference of graphical models}\label{sec:qbgm}
For integer $p\geq 1$, and $i=1,\ldots,p$, let $\Yset_i$ be a nonempty subset of $\rset$, and set $\Yset\eqdef \Yset_1\times\cdots\times\Yset_p$, that we assume is equipped with a reference sigma-finite product measure $\rmd y$. We first consider a class of Markov random field distributions $\{f_\omega,\;\omega\in\Omega\}$ for joint modeling of $\Yset$-valued random variables. We discuss several quasi-Bayesian methods for fitting such models, including the proposed method based on the neighborhood selection method of \cite{meinshausen06}. Gaussian graphical models are then discussed in more detail as special case in Section \ref{sec:ggm}.

let $\M_p$ denote the set of all real symmetric $p\times p$ matrices equipped with the inner product $\pscal{A}{B}_{\textsf{F}}\eqdef \sum_{i\leq j}A_{i}B_{ij}$, and norm $\|A\|_{\textsf{F}}\eqdef \sqrt{\pscal{A}{A}_{\textsf{F}}}$. As above, $\M_p^+$ denotes  the subset of $\M_p$ of positive definite matrices. For $i=1,\ldots,p$, and $1\leq j<k\leq p$, let $B_i:\;\Yset_i\to \rset$ and $B_{jk}:\;\Yset_j\times\Yset_k\to\rset$ be non-zero measurable functions that we assume known. From these functions we define a $\M_p$-valued function $\bar B:\;\Yset\to \M_p$ by  
\[(\bar B(y))_{ij}=\left\{\begin{array}{ll} B_i(y_i) & \mbox{ if } i=j,\\
B_{ij}(y_i,y_j) & \mbox{ if } i < j, \\
 B_{ji}(y_j,y_i), & \mbox{ if }j < i.\end{array}\right.\]
  These functions define the parameter space
\[\Omega\eqdef\left\{\omega\in\M_p:\;Z(\omega)\eqdef \int_{\Yset} e^{-\pscal{\omega}{\bar B(y)}_{\textsf{F}}}\rmd y<\infty\right\}.\] 
We assume that $\Yset$ and $\bar B$ are such that $\Omega$ is non-empty, and we consider the exponential family $\{f_\omega,\;\omega\in\Omega\}$ of densities $f_\omega$ on $\Yset$   given by
\begin{equation}\label{eq:f}
f_\omega(y)=\exp\left(-\pscal{\omega}{\bar B(y)}_{\textsf{F}}-\log Z(\omega)\right),\;\;y\in\Yset.\end{equation}

The model $\{f_\omega,\;\omega\in\Omega\}$ can be useful to capture the dependence structure between a set of $p$ random variables taking values in $\Yset$. The version posited in (\ref{eq:f}) can accommodate a mix of discrete and continuous measurements. The functions $B_i$ are typically viewed as describing the marginal behaviors of the observations in the absence of dependence. Whereas the functions $B_{ij}$ govern  the interactions. These marginal and interaction functions are modulated by the parameter $\omega$. More precisely, if $(Y_1,\ldots,Y_p)\sim f_\omega$, then the parameter $\omega$ encodes the conditional independence structure among the $p$ variables $(Y_1,\ldots,Y_p)$. In particular for $i\neq j$, $\omega_{ij}=0$ means that  $Y_i$ and $Y_j$ are conditionally independent given all other variables. The random variables $(Y_1,\ldots,Y_p)$ can then be represented by an undirect graph where there is an edge between $i$ and $j$ if and only if $\omega_{ij}\neq 0$. This type of models are very useful in practice to tease out direct and indirect connections between sets of random variables.

\begin{example}[Gaussian graphical models]
One recovers the Gaussian graphical model by taking $\Yset_i=\rset$, $B_i(x)=x^2/2$, $B_{ij}(x,y)=xy$, $i<j$. In this case $\Yset=\rset^p$ equipped with the Lebesgue measure, and $\Omega=\M_p^+$.
\end{example}

\begin{example}[Potts models]
For integer $M\geq 2$, one recovers the $M$-states Potts model by taking $\Yset_i=\{1,\ldots,M\}$, $B_i(x)=x$, and $B_{ij}(x,y)=\textbf{1}_{\{x=y\}}$. In this case, $\Yset=\{1,\ldots,M\}^d$ equipped with the counting measure. Since $\Yset$ is a finite set, we have $\Omega=\M_p$. An important special case of the Potts model is a version of the Ising model where $M=2$.
\end{example}

Beyond these examples commonly used in the statistics literature, the family $\{f_\omega,\;\omega \in\Omega\}$ also include the class of auto-models proposed by J. Besag (\cite{besag74}), the mixed discrete-continuous graphical models proposed in  (\cite{cheng:etal:13,yang:etal:14}), as well as few other models used in machine learning, such as Boltzmann machines and Hopefield models (\cite{ruslan:hinton:09}).

Suppose that we observe data $y^{(1)},\cdots,y^{(n)}$ where $y^{(i)}=(y^{(i)}_1,\ldots,y^{(i)}_p)'\in \Yset$ is viewed as a column vector. We set $x\eqdef[y^{(1)},\ldots,y^{(n)}]'\in\rset^{n\times p}$. Given a prior distribution $\Pi$ on $\Omega$, and given the data $x$, the resulting posterior distribution for learning $\omega$ is 
\[\Pi_n(A\vert x) =\frac{\int_A\prod_{i=1}^n f_\omega(y^{(i)}) \Pi(\rmd \omega)}{\int_{\Omega}\prod_{i=1}^n f_\omega(y^{(i)}) \Pi(\rmd \omega)},\;\;A\subseteq \Omega.\]
However, and as discussed in the introduction,  posterior distributions from Markov random fields are typically doubly-intractable\footnote{a terminology introduced by \cite{murray2006mcmc} to mean that the expression of the distribution depends on terms (typically normalizing constants) that cannot be explicitly computed}. There has been some recent advances in MCMC methodology to deal with doubly-intractable distributions (see \cite{lyneetal14} and the references therein). However most of these MCMC algorithms do not scale well to high-dimensional parameter spaces. 

In the frequentist literature, a commonly used approach to circumvent  computational difficulties  with graphical models consists in replacing the likelihood function by a pseudo-likelihood function. For $\omega\in\M_p$, let $\omega_{\cdot i}$ denote the $i$-th column of $\omega$. Note that in the present case, if $(Y_1,\ldots,Y_p)\sim f_\omega$, then for $1\leq j\leq p$, the conditional distribution of $Y_j$ given $\{Y_{k},\; k\neq j\}$  depends on $\omega$ only through the $j$-th column $\omega_{\cdot j}$. We write this conditional distribution as $u\mapsto f^{(j)}_{\omega_{\cdot j}}(u\vert y_{-j})$, where for $y\in\Yset$, $y_{-j}\eqdef (y_1,\ldots,y_{j-1},y_{j+1},\ldots,y_p)$, (with obvious modifications when $j=1,p$).  Let 
\begin{multline*}
\tilde \Omega\eqdef\left\{\omega\in\M_p:\; u\mapsto f^{(j)}_{\omega_{\cdot j}}(u\vert y_{-j}) \mbox{ is a well-defined density on } \Yset_j,\; \right.\\
\left.\mbox{ for all } y\in\Yset,\; \mbox{ and all }1\leq j\leq p\right\}.\end{multline*}
Note that $\Omega\subseteq\tilde\Omega$. The most commonly used pseudo-likelihood method consists in replacing the initial likelihood contribution $f_\omega(y^{(i)})$ by 
\begin{equation}\label{pseudo_f}
\tilde f_\omega(y^{(i)}) = \prod_{j=1}^p f^{(j)}_{\omega_{\cdot j}}(y_j^{(i)}\vert y_{-j}^{(i)}),\;\;\omega\in\tilde \Omega.\end{equation}
This pseudo-likelihood approach, which can be viewed as replacing the model $\{f_\omega,\;\omega\in\Omega\}$ by the pseudo-model $\{\tilde f_\omega,\;\omega\in\tilde\Omega\}$,  typically brings important simplifications. For instance, in the Gaussian case, the parameter space $\tilde \Omega$ corresponds to the space of symmetric matrices with positive diagonals elements, which has a simpler geometry compared to $\M_p^+$. And in the case of discrete graphical models, the conditional models typically have tractable normalizing constants. Despite the fact that $\{\tilde f_{\omega},\;\omega\in\tilde\Omega\}$ is not a proper statistical model, the quasi-likelihood function $\omega\mapsto \sum_{i=1}^n\log \tilde f_{\omega}(y^{(i)})$ still typically leads to consistent estimates of the parameter. The idea goes back to \cite{besag74}, and penalized versions of pseudo-likelihood functions have been employed by several authors to fit high-dimensional graphical models.

A closely related idea is the generalized method of moments (GMM). Given $\omega\in\Omega$, $1\leq j\leq p$, and $y\in\Yset$, define
\[m^{(j)}(\omega_{\cdot j};y_{-j})\eqdef \int_{\Yset_j} u f^{(j)}_{\omega_{\cdot j}}(u\vert y_{-j})\rmd u.\]
Suppose for instance that these conditional moments are well-defined and available in closed form. Then one can derive another pseudo-model $\{\bar f_{\omega},\;\omega\in\bar\Omega\}$, by taking  
\begin{multline*}
\bar \Omega\eqdef\left\{\omega\in\M_p:\; \int_{\Yset_j} |u| f^{(j)}_{\omega_{\cdot j}}(u\vert y_{-j})\rmd u<\infty \mbox{ for all } y\in\Yset,\; \mbox{ and all }1\leq j\leq p\right\}.\end{multline*}
and
\begin{equation}\label{gmm}
\bar f_{\omega}(y^{(i)}) \eqdef  \exp\left(-\sum_{j=1}^p \frac{1}{2\sigma_j^2}\left(y^{(i)}_j-m^{(j)}(\omega_{\cdot j};y^{(i)}_{-j})\right)^2\right),\end{equation}
for positive constants $\sigma_j^2$, $j=1,\ldots,p$. We note that if all the conditional moments of densities in $\{f_\omega,\;\omega\in\Omega\}$ are well defined, then $\Omega\subseteq \bar\Omega$. The function (\ref{gmm}) is the GMM objective function associated with the moment restrictions 
\[\PE_\omega\left[Y_j - m^{(j)}(\omega_{\cdot j};Y_{-j})\right]=0,\;j=1,\ldots,p.\]
In the Gaussian case the two pseudo-likelihood functions (\ref{pseudo_f}) and (\ref{gmm}) coincide. The  moment restriction approach is however more flexible in terms of distributional assumptions.  

Another method for deriving a pseudo-model for this problem is suggested by the neighborhood selection of \cite{meinshausen06}. The idea consists in relaxing the symmetry constraint in $\tilde\Omega$. For $1\leq j\leq p$, we set
\begin{multline*}
\Omega_j\eqdef\left\{\theta\in\rset^p:\; u\mapsto f^{(j)}_{\theta}(u\vert y_{-j}) \mbox{ is a well-defined density on } \Yset_j,\; \right.\\
\left.\mbox{ for all } y\in\Yset,\; \mbox{ and all }1\leq j\leq p\right\}.\end{multline*}
We note that if $\omega\in\Omega$, then $\omega_{\cdot j}\in\Omega_j$. Hence these sets $\Omega_j$ are nonempty, and we define $\check\Omega\eqdef \Omega_1\times \cdots\times \Omega_p$, that we identify as a subset of the space of $p\times p$ real matrices $\rset^{p\times p}$. In particular if $\omega\in\check\Omega$, and  consistently with our notation above, $\omega_{\cdot,j}$ denotes the $j$-column of $\omega$.  We consider the  pseudo-model $\{\check f_\omega,\;\omega\in\check\Omega\}$, where
\begin{equation}\label{quasi:lik:check}
\check f_{\omega}(y)\eqdef \prod_{j=1}^p f^{(j)}_{\omega_{\cdot j}}(y_j\vert y_{-j}),\;\;\omega\in \check\Omega,\;\;y\in\Yset.
\end{equation}
Notice that by definition $\check\Omega$ is a product space, whereas $\tilde\Omega$ is not, due to the symmetry constraint. This implies that $\omega\mapsto\check f_{\omega}(y)$ factorizes along the columns of $\omega$, whereas $\omega\mapsto\tilde f_{\omega}(y)$ typically does not. One can then maximize a penalized version of $\omega\mapsto\sum_{i=1}^n \log \check f_{\omega}(y^{(i)})$, and this corresponds to the neighborhood selection method of (\cite{meinshausen06}, see also \cite{sun:zhang:13}). The optimization can be advantageously solved in parallel for each component if the penalty is separable.

As it turns out, all these pseudo-models can also be used in the Bayesian framework, as shown by the seminal work of \cite{chernozhukov:hong03}. We shall focus on the pseudo-model (\ref{quasi:lik:check}). With a prior distribution $\Pi$ on $\check\Omega$, the quasi-likelihood function $\omega\mapsto\check f_\omega$ leads to a quasi-posterior distribution given by
\begin{equation*}
\check\Pi_{n,p}(A\vert x) = \frac{\int_{A} \prod_{i=1}^n\check f_{\omega}(y^{(i)})\Pi(\rmd \omega)}{\int_{\check\Omega} \prod_{i=1}^n\check f_{\omega}(^{(i)})\Pi(\rmd \omega)},\;\;A\subset\check\Omega,
\end{equation*}
for which MCMC algorithms can be constructed. 
Let us assume that the prior distribution factorizes: $\Pi(\rmd\omega)=\prod_{j=1}^p \Pi_j(\omega_{\cdot j})$. Then  we are led to the  quasi-posterior distribution 
\begin{equation}\label{post:mrf}
\check\Pi_{n,p}(\rmd u_1,\cdots \rmd u_p\vert x) = \prod_{j=1}^p \check\Pi_{n,p,j}(\rmd u_j\vert x),
\end{equation}
where
\[\check\Pi_{n,p,j}(\rmd u\vert x) = \frac{\prod_{i=1}^n f^{(j)}_{u}(y^{(i)}_j\vert y^{(i)}_{-j})\Pi_j(\rmd u)}{\int_{\Omega_j} \prod_{i=1}^n f^{(j)}_{u}(y^{(i)}_j\vert y^{(i)}_{-j})\Pi_j(\rmd u)},\]
is a probability measure on $\Omega_j$. Basically, relaxing the symmetry allows us to factorize the quasi-likelihood function and this leads to a factorized quasi-posterior distribution, as in (\ref{post:mrf}). Each component of this quasi-posterior distribution can then be explored independently. Despite its simplicity, when used in a parallel computing environment, this approach increases by one order of magnitude the size of graphical models that can be estimated. 

\begin{remark}[symmetrization and positive definiteness]
One of the limitation of the quasi-Bayesian approach outlined above is that the distribution $\check\Pi_{n,p}$ does not necessarily produce symmetric and positive definite matrices. However, because of the contraction properties of $\check\Pi_{n,p}$ discussed below, when the true precision matrix is well conditioned, typical realizations of $\check\Pi_{n,p}$ are actually symmetric and positive definite, with high probability. From a practical viewpoint, one can remedy a broken symmetry using various symmetrization rules as suggested for instance in \cite{meinshausen06}. Lack of positive definiteness is more expensive to repair, but can be addressed for instance by projection of the convex cone of semipositive definite matrices via eigendecomposition (\cite{higham:88}), and by addition of a small diagonal matrix.
\end{remark}

\section{Gaussian graphical models}\label{sec:ggm}
We now specialize the discussion to the Gaussian case, where $\textsf{Y}_i=\rset$, $B_i(x)=x^2/2$, and $B_{ij}(x,y)=xy$. Hence in  this case, $\Omega = \M_p^+$, $\tilde \Omega$ corresponds to the set of symmetric matrices with positive diagonal elements, which is an important simplification over $\M_p^+$. Further dropping the symmetry leads to $\check\Omega$, which here is the space of $p\times p$ real matrices with positive diagonal. Assuming that the diagonal elements are known and given, we shall identify $\check\Omega$ with the matrix space $\rset^{(p-1)\times p}$.

If $\vartheta\in\M_p^+$, and $(Y_1,\ldots,Y_p)\sim f_\vartheta$, it is well known that for all $j\in\{1,\ldots,p\}$, the conditional distribution of $Y_j$ given $Y_k=y_k$, for $k\neq j$ is
\begin{equation}\label{cond:dist}
\textbf{N}\left(-\sum_{k\neq j}\frac{\vartheta_{kj}}{\vartheta_{jj}}y_k,\frac{1}{\vartheta_{jj}}\right),
\end{equation}
where $\textbf{N}(\mu,\sigma^2)$ denotes the Gaussian distribution with mean $\mu$ and variance $\sigma^2$.  Given data $x\in\rset^{n\times p}$, given $\sigma_j^2>0$, and given these conditional distributions, the product of the quasi-model (\ref{quasi:lik:check}) across the data set gives (upto normalizing constants that we ignore) the quasi-likelihood
\begin{multline}\label{quasi:lik:GGM}
q(\theta;x) \eqdef \prod_{j=1}^p q_{j}(\theta_{\cdot j};x),\;\\
\mbox{ with }\; q_j(\theta_{\cdot j};x) \eqdef \exp\left(-\frac{1}{2\sigma_j^2}\|x_{\cdot j} - x^{(j)}\theta_{\cdot j}\|_2^2\right),\;\;\theta\in\rset^{(p-1)\times p},\end{multline}
where $x^{(j)}\in\rset^{n\times(p-1)}$ is the matrix obtained from $x$ by removing the $j$-th column, and $x_{\cdot j}$  (resp. $\theta_{\cdot j}$)  denotes the $j$-column of $x$ (resp. $\theta$). Given (\ref{cond:dist}), it is clear that $\sigma_j^2$ is a proxy for $1/\vartheta_{jj}$.  It is also clear that maximizing (\ref{quasi:lik:GGM}) or a penalized version thereof would give an estimate of $-(\vartheta_{kj}/\vartheta_{jj})_{k\neq j}$. This is precisely the idea of the neighborhood selection of \cite{meinshausen06}, or the sparse matrix inversion with scaled lasso of \cite{sun:zhang:13}. These methods can be used to recover the sign (the structure) of $\vartheta$, but also gives an estimate of $\vartheta_{kj}$ if $\sigma_j^2$ is a good estimate of $1/\vartheta_{jj}$, or if $\vartheta_{jj}$ can also be estimated (as in the case of the scaled-lasso).  We combine (\ref{quasi:lik:GGM}) with a prior distribution $\Pi(\rmd \theta)=\prod_{j=1}^p \Pi_j(\rmd\theta_{\cdot j})$ to obtain a  quasi-posterior distribution on $\rset^{(p-1)\times p}$ given by
\begin{equation}\label{quasi:post:GGM1}
\check\Pi_{n,p}(\rmd \theta\vert x)  =  \prod_{j=1}^p \check\Pi_{n,p,j}(\rmd \theta_{\cdot j}\vert x,\sigma_j^2),
\end{equation}
where $\check\Pi_{n,p,j}(\cdot \vert x,\sigma_j^2)$ is the probability measure on $\rset^{p-1}$ given by
\begin{equation*}
\check\Pi_{n,p,j}(\rmd z\vert x,\sigma_j^2)\propto  q_j(z;x)\Pi_j(\rmd z).\end{equation*}
Again the main appeal of $\check\Pi_{n,p}$ is its factorized form, which implies that Monte Carlo samples from $\check\Pi_{n,p}$ can be obtained by sampling in parallel from the $p$ distributions $\check\Pi_{n,p,j}$.

\subsection{Prior distribution}\label{sec:prior}
We address here the choice of the prior distribution $\Pi_j$. Since we are dealing with a linear regression problem, there are many possible ways to set up the prior. We advocate the use of discrete-continuous mixture distributions because these prior distributions have well-understood posterior contraction properties (\cite{castillo:etal:14,atchade:15b}), and are known to produce sparse posterior samples. 

For each $j\in\{1,\ldots,p\}$, we build the prior $\Pi_j$ on $\rset^{(p-1)}$ as in  \cite{atchade:15b}. First, let  $\Delta_{p}\eqdef \{0,1\}^{p-1}$, and let $\{\pi_\delta,\;\delta\in\Delta_{p}\}$ denote a discrete probability distribution on $\Delta_{p}$ (which we assume to be the same for all the components $j$). We take $\Pi_j$ as the distribution of the random variable $u\in\rset^{p-1}$ obtained as follows. 
\begin{multline}\label{prior:ggm}
\delta\sim \{\pi_\delta,\;\delta\in\Delta_p\}.\;\;\mbox{Given }\; \delta,\; (u_1,\ldots,u_{p-1})\;\mbox{ are conditionally independent }\\
\mbox{ and }\;u_k\vert \delta\sim \left\{\begin{array}{ll} \textsf{Dirac}(0) & \mbox{ if } \delta_k=0\\ \textsf{EN}(\rho_{1j},\rho_{2j}) & \mbox{ if } \delta_k=1\end{array}\right.,\end{multline}
where $\textsf{Dirac}(0)$ is the Dirac measure on $\rset$ with mass at $0$, and  $\textsf{EN}(\rho_{1j},\rho_{2j})$ denotes the elastic net distribution on $\rset$ with density proportional to \begin{equation}\label{rep:EN}
z\mapsto \frac{1}{C_j}\exp\left(-\alpha\frac{\rho_{1j}}{\sigma_j^2}|z| -(1-\alpha)\frac{\rho_{2j}}{\sigma_j^2}\frac{z^2}{2}\right), \;\;z\in\rset,\end{equation}
for parameters $\rho_{1j},\rho_{2j}>0$, and where $\alpha\in (0,1]$ is a fixed parameter (in the simulation we use $\alpha=0.9$). The term $C_j$ is the normalizing constant\footnote{can be explicitly computed as $C_j = C_\alpha\left(\frac{\rho_{1j}}{\sigma_j^2},\frac{\rho_{2j}}{\sigma_j^2}\right)$, where, with $\Phi$ denoting the cdf of standard normal distribution, $\textsf{erfcx}(x)= 2e^{x^2}\Phi(-\sqrt{2}x)$ denoting the scaled complementary error function,
\[C_\alpha(\lambda_1,\lambda_2) \eqdef \left\{\begin{array}{ll}\sqrt{\frac{2\pi}{(1-\alpha)\lambda_2}} \textsf{erfcx}\left(\frac{\alpha\lambda_1}{\sqrt{2(1-\alpha)\lambda_2}}\right) & \mbox{ if } \alpha\in[0, 1)\\ \frac{2}{\lambda_1} & \mbox{ if } \alpha=1\end{array}\right..\]
}. We use a fully-Bayesian approach for selecting $\rho_{1j},\rho_{2j}$. More precisely, we assume that $\rho_{1j}$ and $\rho_{2j}$ have independent prior distribution $\rho_{1j}\sim \phi$, $\rho_{2j}\sim \phi$, where $\phi$ is the uniform distribution $\textbf{U}(a_1,a_2)$ for $a_1=10^{-5}$, and where the choice of $a_2$ follows \cite{atchade:15a}~Section 4. 

We focus on situations where, although $p$ is possibly large, the undirected graph defined by the underlying precision matrix is sparse. This prior information is encoded in the prior distribution, by choosing $\pi_\delta$ as follows. We assume that the components of $\delta$ are conditionally independent with distribution $\textbf{Ber}(\textsf{q})$ given $\textsf{q}$, where $\textsf{q}\sim \textbf{Beta}(1,p^u)$, for some $u>1$. Hence according to the prior distribution, the proportion of non-zero component of each column of $\theta$ is $1/p^{u-1}$. We use $u=1.5$.

With the prior distribution given above, and given $\sigma_j^2$, we obtain a fully specified quasi-posterior distribution 
\begin{equation}\label{quasi:post:GGM2}
\prod_{j=1}^p \bar \Pi_{n,p,j}(\delta,\rmd\theta,\rmd\textsf{q},\rmd \rho_{1j},\rmd \rho_{2j}\vert x,\sigma^2_j),\end{equation}
where the $j$-th component $\bar\Pi_{n,p,j}(\cdot \vert x,\sigma_j^2)$ can be written as follows. For $\delta\in\Delta_{p}$, let $\mu_{\delta}$ be the product measure on $\rset^{p-1}$ defined as $\mu_{\delta}(\rmd u)=\prod_{j=1}^{p-1}\nu_{\delta_{j}}(\rmd u_{j})$, where  $\nu_{0}(\rmd z)$ is the Dirac mass at $0$, and $\nu_1(\rmd z)$ is the Lebesgue measure on $\rset$. Then
\begin{multline}\label{post:ggm:full}
\bar\Pi_{n,p,j}(\delta,\rmd\theta,\rmd\textsf{q},\rmd \rho_{1j},\rmd \rho_{2j}\vert x,\sigma^2_j)\propto q_j(\theta;x)\left(\frac{\textsf{q}}{1- \textsf{q}}\right)^{\|\delta\|_1}(1-\textsf{q})^{d+d^{u}-1}\\
\times \left(\frac{1}{C_j}\right)^{\|\delta\|_1}e^{-\alpha\frac{\rho_{1j}}{\sigma_j^2}\|\theta\|_1 -(1-\alpha)\frac{\rho_{2j}}{\sigma_j^2}\frac{\|\theta\|_2^2}{2}}\phi(\rho_{1j})\phi(\rho_{2j})\mu_{\delta}(\rmd\theta)\rmd \textsf{q}\rmd\rho_{1j}\rmd\rho_{2j}.
\end{multline}

Notice that if, instead of the uniform prior distribution $\phi$, we use a point mass prior distribution for $\rho_{1j}$ and $\rho_{2j}$ in (\ref{quasi:post:GGM2}), and integrate out $\textsf{q}$ and $\delta$, we recover exactly (\ref{quasi:post:GGM1}) where the prior $\Pi_j$ is given by (\ref{prior:ggm}) and (\ref{rep:EN}).  The quasi-posterior distribution (\ref{post:ggm:full}) depends on the choice of $\sigma_j^2$. Ideally we would like to set $\sigma_j^2=1/\vartheta_{jj}$. However this quantity is unknown. In practice, we suggest choosing $\sigma_j^2$ by empirical Bayes, following \cite{atchade:15a}~Section 4.3. We explore this approach in the simulations.

\subsection{Approximate MCMC simulation}\label{sec:mcmc}
Given $j\in\{1,\ldots,p\}$, sampling from the distribution $\bar\Pi_{n,p,j}(\cdot\vert x)$ given in (\ref{post:ggm:full}) is a difficult computation task, due to a lack of smoothness in $\theta$, and its trans-dimensional nature\footnote{for two different elements $\delta,\delta'$ of $\Delta_p$, the probability measures $\check\Pi_{n,p,j}(\delta,\cdot\vert x,\sigma_j^2)$ and $\check\Pi_{n,p,j}(\delta',\cdot\vert x,\sigma_j^2)$ are mutually singular}. Here we follow the approach developed by the author in \cite{atchade:15a}, which produces approximate samples from (\ref{post:ggm:full}) by sampling from its Moreau-Yosida approximation $\bar\Pi_{n,p,j}^{(\gamma)}(\delta,\rmd\theta,\rmd\textsf{q},\rmd \rho_{1j},\rmd \rho_{2j}\vert x,\sigma^2_j)$. The parameter $\gamma\in (0,1/4]$ controls the quality of the approximation.  It is shown (\cite{atchade:15a}~Theorem 5) that $\bar\Pi_{n,p,j}^{(\gamma)}(\cdot\vert x,\sigma^2_j)$ converges weakly to $\bar\Pi_{n,p,j}(\cdot\vert x,\sigma^2_j)$ as $\gamma\to 0$. The idea of working with the Moreau-Yosida approximation instead of the distribution itself is attractive because for $\gamma>0$ fixed, all the probability measures $\bar\Pi_{n,p,j}^{(\gamma)}(\delta,\cdot\vert x,\sigma^2_j)$ for $\delta\in\Delta_p$ are smooth and have densities with respect to the (same) Lebesgue measure $\rmd\theta\rmd \textsf{q}\rmd\rho_{1j}\rmd\rho_{2j}$. As a result, one can sample easily from $\bar\Pi_{n,p,j}^{(\gamma)}(\cdot\vert x,\sigma^2_j)$ without any need for trans-dimensional MCMC techniques. 

\subsection{Posterior contraction and rate}
Despite the fact that the number of columns ($p$) and the dimension of each column ($p-1$) are both increasing, we will show next that for $n$ reasonably large and for a well-behaved underlying distribution, typical realizations of the quasi-posterior distribution $\check\Pi_{n,d}$ given in (\ref{quasi:post:GGM1}) put most of its probability mass on small neighborhoods of the true value of the parameter. 

Given a random sample $X\in\rset^{n\times p}$, we shall study the behavior of the random probability measure $\check\Pi_{n,p}(\rmd \theta\vert X)$ on $\rset^{(p-1)\times p}$ as given in (\ref{quasi:post:GGM1}), for large $n,p$. 
We focus on the case where the prior distribution is given by (\ref{prior:ggm})-(\ref{rep:EN}), with $\alpha=1$ (hence $\rho_{2j}$ is irrelevant), and $\rho_{1j}$ fixed. The choice $\alpha=1$ corresponds to the \textsf{Laplace} prior ($\ell^1$ prior), and is made here mainly for simplicity.  We assume below that the rows of $X$ are i.i.d. random variables from a mean-zero Gaussian distribution with precision matrix $\K$.

\begin{assumption}\label{A1}
For some $\K\in\M_p^+$, $X=Z\K^{-1/2}$, where $Z\in\rset^{n\times p}$ is a random matrix with i.i.d. standard normal entries.
\end{assumption}

From the true precision matrix $\K$, we now form the true value of the parameter $\theta_\star\in\rset^{(p-1)\times p}$ towards which $\check\Pi_{n,p}$ should converge. For $j=1,\ldots,p$, $\theta_{\star kj}=-\K_{kj}/\K_{jj}$, for $k=1,\ldots,j-1$, and $\theta_{\star kj}=-\K_{(k+1)j}/\K_{jj}$, for $k=j,\ldots,p-1$. Let $\delta_\star\in\{0,1\}^{(p-1)\times p}$ be the sparsity structure of $\theta_\star$, defined as $\delta_{\star kj}=\textbf{1}_{\{|\theta_{\star kj}|>0\}}$. We set
\begin{equation*}
s_{\star j}\eqdef \sum_{k=1}^{p-1} \textbf{1}_{\{|\theta_{\star kj}|>0\}},\;\;j=1,\ldots,p\;\;\mbox{ and }\;\; s_\star \eqdef \max_{1\leq j\leq p}s_{\star j}.\]
Hence $s_{\star j}$ is the degree of node $j$, and $s_\star$ is the maximum node degree in the undirected graph defined by $\K$.  

The asymptotic behavior of $\check\Pi_{n,p}$ depends crucially on certain restricted and $m$-sparse eigenvalues of the true precision matrix $\K$, that we introduce next. We set
\begin{equation}\label{u_kappa:1}
\underline{\kappa} \eqdef \inf_{\delta\in\{0,1\}^p:\;\|\delta\|_0\leq s_\star}\inf \left\{\frac{u'\K u}{\|u\|_2^2}:\;u\in\rset^{p},\;u\neq 0,\mbox{ s.t. }\sum_{k:\; \delta_{k}=0}|u_k|\leq 7\sum_{k:\; \delta_{k}=1}|u_k|\right\},\end{equation}
and for $1\leq s\leq p$, 
\begin{multline}\label{u_kappa:2}
\utilde{\kappa}(s)  \eqdef \inf \left\{\frac{u'\K u}{\|u\|_2^2}:\;u\in\rset^{p},\;1\leq \|u\|_0\leq s\right\},\\
\;\;\tilde{\kappa}(s)  \eqdef \sup \left\{\frac{u'\K u}{\|u\|_2^2}:\;u\in\rset^{p},\;1\leq \|u\|_0\leq s\right\}.
\end{multline}
In the above equations, we convene that $\inf\emptyset=+\infty$, and $\sup\emptyset=0$.  We shall make the following assumption on the prior distribution $\Pi$ on $\rset^{(p-1)\times p}$. 

\begin{assumption}\label{A2}
$\Pi(\rmd\theta)=\prod_{j=1}^p\Pi_j(\rmd\theta_{\cdot j})$, where for each $j\in\{1,\ldots,p\}$ $\Pi_j$ is of the form described in (\ref{prior:ggm}), with $\alpha_j=1$, and
\begin{equation}\label{rho:ggm}
\rho_{1j} =\rho_j \eqdef \sqrt{\frac{54\tilde\kappa(1)}{\K_{jj}}n\log(p)}.
\end{equation}
Furthermore, the distribution $\{\pi_\delta,\;\delta\in\Delta_{p}\}$ satisfies: $\pi_\delta = g_{\|\delta\|_0} {p-1\choose \|\delta\|_0}^{-1}$, for a discrete distribution $\{g_s,\;0\leq s\leq p-1\}$, for which there exist positive universal constant $c_1,c_2, c_3, c_4$  such that
\begin{equation}\label{cond:g}
\frac{c_1}{p^{c_3}} g_{s-1} \leq g_s\leq \frac{c_2}{p^{c_4}} g_{s-1},\;s = 1,\ldots,p-1.\end{equation}
\end{assumption}

\begin{remark}\label{rem:pi:delta}
This class of prior distributions was pioneered by \cite{castillo:etal:14}. The example of $\{\pi_\delta\}$ presented in Section \ref{sec:prior} satisfies (\ref{cond:g})  with $c_1=1/2$, $c_2=1$, $c_3 = u$ and $c_4=u-1$.
\end{remark}

Our first result shows that if a minimum sample size requirement is met, and if $\K$ is well-behaved, then typical realizations of  $\theta\sim\check\Pi_{n,p}(\cdot\vert X)$ are sparse, with sparsity structure close to the sparsity structure of $\theta_\star$.

\begin{theorem}\label{thm1}
Assume H\ref{A1} and H\ref{A2}, with $\underline{\kappa}>0$. Suppose also that $p$ is large enough so that $p^{c_4}\geq 8c_2\max(1,2c_2)$, with $c_2,c_4$ as in H\ref{A2}. For $1\leq j\leq p$, set 
\[\zeta_j = \frac{4}{c_4} + s_{\star j} + \frac{2}{c_4}\left(\frac{\log(4e p)}{\log(p)} + \frac{6912}{\sigma_j^2\K_{jj}}\frac{\tilde\kappa(1)}{\underline{\kappa}} + \frac{\sigma_j^2\K_{jj}}{24(\log(p))^2}\frac{\tilde\kappa(s_\star)}{\tilde\kappa(1)}\right) s_{\star j}.\]
Then there exist universal finite constants $a_1>0,\;a_2>0$ such that if
\begin{equation}\label{eq:sample:size:thm1}
n\geq a_1 s_\star\left(1+\frac{\tilde\kappa(1)}{\underline{\kappa}}\right)\log(p),\end{equation}
\[\PE\left[\check\Pi_{n,p}\left(\left\{\theta\in\rset^{(p-1)\times p}:\;\|\theta_{\cdot j}\|_0\geq \zeta_j \mbox{ for some } j \right\}\;\vert X\right)\right]\leq 2\left(\frac{1}{e^{a_2 n}} +\frac{2}{p}\right).\]
\end{theorem}
\begin{proof}
See Section \ref{sec:proof:thm1}.
\end{proof}

\begin{remark}
In the ideal case where $\sigma^2_j=1/\K_{jj}$, and for $n,p$ large, we see that if $\theta\sim\check\Pi_{n,p}(\cdot\vert X)$, then with high probability $\|\theta_{\cdot j}\|_0<\zeta_j$ for all $j\in\{1,\ldots,p\}$, and
\[\zeta_j \approx \frac{4}{c_4} + s_{\star j} + \frac{2}{c_4}\left(1+  6912\frac{\tilde\kappa(1)}{\underline{\kappa}}\right)s_{\star j}.\]
Hence if the restricted condition number $\tilde\kappa(1)/\underline{\kappa}$ remains small, then for large values of $p$, typical realizations of $\check\Pi_{n,p}$ are sparse. The large constant $6912$ appearing in the theorem is most likely an artifact of the techniques used in the proof, and can probably be improved.
\end{remark}

\medskip
For $\theta\in\rset^{(p-1)\times p}$, we set
\[\nnorm{\theta} \eqdef \max_{1\leq j\leq p} \|\theta_{\cdot j}\|_2.\]
\begin{theorem}\label{thm2}
Assume H\ref{A1}, H\ref{A2}, with $\underline{\kappa}>0$. Let $\zeta_j$ be as in theorem \ref{thm1}. Set $\bar s\eqdef \max_{1\leq j\leq p} \bar s_j$, where $\bar s_j \eqdef s_{\star j} +\zeta_j$ if $s_{\star j}>0$, and $\bar s_j\eqdef 1$ otherwise. Set
\[\epsilon \eqdef 12\sqrt{6}\frac{\sqrt{\tilde\kappa(1)}}{\utilde\kappa(\bar s)}\sqrt{\frac{\bar s\log(p)}{n}},\]
and $M_0\eqdef  \max\left(96,(4+c_4(2+c_3)/2)\max_j\sigma_j^2\vartheta_{jj}\right)$. Then there exists universal finite constant $a_1>0$, $a_2>0$ such that if $p\geq \max(24e,2/c_1)$, $p^{c_4}\geq 8c_2\max(1,2c_2)$, and 
\begin{equation}\label{eq:sample:size:thm2}
n\geq a_1 s_\star\left(\frac{\bar\kappa(1)}{\underline{\kappa}}\right)\log(p),\;\mbox{ and }\; n\geq a_1 \bar s \log(p),\end{equation}
then
\[\PE\left[\check\Pi_{n,p}\left(\left\{\theta\in\rset^{(p-1)\times p}:\;\nnorm{\theta-\theta_\star}>M_0\epsilon\right\}\vert X\right)\right]\leq 3\left(\frac{1}{e^{a_2n}} + \frac{4}{p}\right).\]
\end{theorem}
\begin{proof}
See Section \ref{sec:proof:thm2}.
\end{proof}

\medskip
\begin{remark}
Theorem \ref{thm2} shows that the contraction rate of $\check\Pi_{n,p}$ towards $\theta_\star$ in the $\nnorm{\cdot}$ norm is $O(\epsilon)$. This corresponds to the worst rate among the rates of contraction of the $p$ linear regression problems performed during  the neighborhood selection procedure. This rate is similar to the rate of convergence of the (frequentist) neighborhood selection method of \cite{meinshausen06}, which is of order 
\begin{equation}\label{rate:ns}
 \sqrt{\frac{s_\star\log(p)}{n}},\end{equation}
(see the discussion in Section 3.4 of \cite{ravikumaretal11}). The main difference between (\ref{rate:ns}) and the rate in Theorem \ref{thm2} is in the dependence on the maximum degree $s_\star$. In the Bayesian case, $s_\star$ is replaced by a worst-case estimate from Theorem \ref{thm1}, namely the largest value that the maximum degree of $\theta\sim \check\Pi_{n,p}(\cdot\vert X)$ can take (with a significant probability).

An interesting difference pointed out in \cite{ravikumaretal11} (Section 3.4), between the neighborhood selection approach and graphical lasso approaches, is that neighborhood selection methods requires a sample size $n$ that scales linearly in $s_\star$, whereas graphical lasso methods require a sample size sample that scales quadratically in $s_\star$. We recover the same dependence on $n$ in Equations (\ref{eq:sample:size:thm1}) and (\ref{eq:sample:size:thm2}) of Theorems \ref{thm1} and \ref{thm2}, where the sample size scales linearly in $\bar s$.

\end{remark}

\section{Numerical experiments}\label{sec:example}
We evaluate the behavior of the quasi-posterior distribution (\ref{quasi:post:GGM1}) on three simulated datasets. As benchmark, we also report the results obtained using  the elastic net estimator
\[\hat\vartheta_{\textsf{glasso}} = \textsf{Argmin }_{\theta \in\M_p^+} \left[-\log\det\theta + \textsf{Tr}(\theta S) + \lambda\sum_{i,j} \left(\alpha | \theta_{ij} | + \frac{(1-\alpha)}{2} \theta_{ij}^2 \right)\right],\]
where $S=(1/n)x'x$, $\alpha=0.9$, and $\lambda>0$ is a regularization parameter. We choose $\lambda$ by minimizing $-\log\det\left(\hat\theta(\lambda)\right)+ \textsf{Tr}(\hat\theta(\lambda) S) +\log(n) \sum_{i<j} \textbf{1}_{\{|\hat\theta(\lambda)_{ij}|>0\}}$, over a finite set of values of $\lambda$.  We hasten to add that our goal is not to compare the quasi-Bayesian method to graphical \texttt{lasso}, since the former utilizes vastly more computing power that the latter. The outputs are also very different, since \texttt{Glasso} gives only a point estimate whereas the Bayesian approach produces a full posterior distribution. Rather, we report these numbers as references, to help the reader better understand the behavior of the proposed methodology.
 
\subsection{Simulation set ups}
We generate a data matrix $x\in\rset^{n\times p}$ with i.i.d. rows from $\textbf{N}(0,\vartheta^{-1})$, $\vartheta\in\M_p^+$. Throughout we set the sample size to $n=250$, and we consider three settings. 
\begin{description}
\item [(a)] $\vartheta$ is generated as in Setting (c) below, but using $p=100$ nodes.
\item [(b)] In this case $p=500$, and we take $\vartheta$ from the \textsf{R}-package \texttt{space}   based on the work \cite{peng:etal:09}\footnote{The precision matrix used here corresponds to the example ``Hub network" in Section 3 of \cite{peng:etal:09}. A non-sparse version of $\vartheta$ is attached to the \texttt{space} package}. These authors have designed a precision matrix $\vartheta$ that is modular with $5$ modules of $100$ nodes each. Inside each module, there are $3$ hubs with degree around $15$, and $97$ other nodes with degree at most $4$. The total number of edges is $587$. The resulting  partial correlations fall within $(-0.67,-0.10]\cup [0.10,0.67)$. As explained in \cite{peng:etal:09}, this type of networks are useful models for biological networks.
\item [(c)] In this case $p=1,000$, and we build $\vartheta$ as follows. First we generate a symmetric sparse matrix $B$ such that the number of off-diagonal non-zeros entries is roughly $2p$. We magnified the signal by adding $3$ to all the non-zeros entries of $B$ (subtracting $3$ for negative non-zero entries). Then we set $\vartheta=B + (\epsilon-\lambda_{\min}(B))I_p$, where $\lambda_{\min}(B)$ is the smallest eigenvalue of $B$, with $\epsilon=1$. In this example, values of the partial correlations are typically in the range $(-0.46,-0.18]\cup [0.18,0.48)$.
\end{description}

\medskip

To evaluate the effect of the hyper-parameter $\sigma_j^2$, we report two sets of results. One where $\sigma^2_j = 1/\vartheta_{jj}$, and another set of results where $\vartheta_{jj}$ is assumed unknown and we select $\sigma_j^2$ from the data, using the cross-validation estimator described in \cite{reid:etal:13} (see also \cite{atchade:15a}~Section 4.3).

\medskip
In order to mitigate the uncertainty in some of the results reported below, we repeat all the MCMC simulations  20 times. Hence, to summarize, for each setting (a), (b), and (c), we generate one precision matrix $\vartheta$. Given $\vartheta$, we generate 20 datasets, and for each dataset, we run two MCMC samplers (one where the $\sigma^2_j$'s are taken as the $1/\vartheta_{jj}'s$, and one where they are estimated from the data).
\subsection{Estimation details}\label{sec:estimation}
As explained in Section \ref{sec:mcmc}, we first approximate the target quasi-posterior 
\[\prod_{j=1}^p \bar\Pi_{n,p,j}(\delta,\rmd\theta,\rmd\textsf{q},\rmd \rho_{1j},\rmd \rho_{2j}\vert x,\sigma^2_j)\]
by 
\begin{equation}\label{dist:ex}
\prod_{j=1}^ p \bar\Pi^{(\gamma)}_{n,p,j}(\delta,\rmd\theta,\rmd \textsf{q},\rmd \phi_{1j},\rmd\phi_{2j}\vert x,\sigma^2_j),\end{equation}
where $\bar\Pi_{n,p,j}^{(\gamma)}(\cdot\vert x,\sigma^2_j)$ is the Moreau-Yosida approximation of $\bar\Pi_{n,p,j}(\cdot\vert x,\sigma^2_j)$ given in (\ref{post:ggm:full}).  In all the simulations below, we use $\gamma=0.2$. We then sample from (\ref{dist:ex}) by parallel computing, each distribution $\bar\Pi^{(\gamma)}_{n,p,j}(\cdot\vert x,\sigma^2_j)$ at the time, and using the MCMC sampler developed in \cite{atchade:15a}. We use a high-performance computer with $100$ nodes.

\medskip

To simulate from $\bar\Pi^{(\gamma)}_{n,p,j}(\cdot\vert x,\sigma^2_j)$ for a given $j$, we run the MCMC sampler for $50,000$ iterations and discard the first $10,000$ iterations as burn-in. We use Geweke's diagnostic test on the remaining samples to test for convergence using the negative pseudo-log-likelihood function $\theta\mapsto \frac{1}{2\sigma_j^2}\|x_{\cdot j}-x^{(j)}\theta\|_2^2$. All the samplers passed the test. This suggests that $50,000$ is a reasonably large number of iterations for these examples.

From the MCMC output, we estimate the structure $\delta\in\{0,1\}^{p\times p}$ as follows. We set the diagonal of $\delta$ to one, and for each off-diagonal entry $(i,j)$ of $\delta$, we estimate $\delta_{ij}$ as equal to $1$ if the sample average estimate of $\delta_{ij}$ (from the $j$-th chain) and the sample average estimate of $\delta_{ji}$  (from the  $i$-th chain) are both larger than $0.5$. Otherwise $\delta_{ij}=0$. Obviously, other symmetrization rules could be adopted.

Given the estimate $\hat\delta$ say, of $\delta$, we estimate $\vartheta\in\rset^{p\times p}$ as follows. We set the diagonal of $\vartheta$ to $(1/\sigma^2_j)$. For $i\neq j$, if $\hat\delta_{ij}=0$, we set $\vartheta_{ij}=\vartheta_{ji}=0$. Otherwise we estimate $\vartheta_{ij} = \vartheta_{ji}$ as $0.5(-1/\sigma_j^2)\bar\vartheta_{ij}+0.5(-1/\sigma_i^2)\bar\vartheta_{ji}$, where $\bar\vartheta_{ij}$ (resp. $\bar\vartheta_{ji}$) is the Monte Carlo sample average estimate of $\vartheta_{ij}$ from the $j$-th chain (resp. $i$-th chain). For all the off-diagonal components $(i, j)$ such that $\hat\delta_{ij}=1$, we also produce a $95\%$ posterior interval by taking the union of the $95\%$ posterior intervals from the $i$-th and $j$-th chains. When $\hat\delta_{ij}=0$, we set the confidence interval to $\{0\}$.

\subsection{Results}
We look  at the performance of the method by computing the relative Frobenius norm, the sensitivity and the precision of the estimated matrix (as obtained above). These quantities are defined respectively as
\begin{multline}\label{sen:prec}
\mathcal{E} = \frac{\|\hat\vartheta-\vartheta\|_{\textsf{F}}}{\|\vartheta\|_{\textsf{F}}},\;\;\textsf{SEN} = \frac{\sum_{i<j} \textbf{1}_{\{|\vartheta_{ij}|>0\}} \textbf{1}_{\{\textsf{sign}(\hat\vartheta_{ij})=\textsf{sign}(\vartheta_{ij})\}}}{\sum_{i<j} \textbf{1}_{\{|\vartheta_{ij}|>0\}}}\,;\; \\
\mbox{ and }\;\;\textsf{PREC}=\frac{\sum_{i<j} \textbf{1}_{\{|\hat\vartheta_{ij}|>0\}} \textbf{1}_{\{\textsf{sign}(\hat\vartheta_{ij})=\textsf{sign}(\vartheta_{ij})\}}}{\sum_{i<j} \textbf{1}_{\{|\hat\vartheta_{ij}|>0\}}}.\end{multline}
We average these statistics over the 20 simulations replications.  We compute also the same quantities for the elastic net $\hat\vartheta_{\textsf{glasso}}$. These results are reported in Table \ref{table:1}-\ref{table:3}. These results suggest that the quasi-Bayesian procedure has good contraction properties  in the Frobenius norm (and hence in the $L_{\infty,2}$ norm). The results also suggest that the quasi-Bayesian procedure is not very sensitive (it has a high false negative rate), but has excellent precision (it has a very low false positive rate), even with $p=1,000$. The same conclusion seems to hold across all three network settings considered in the simulations. 

Another interesting point to notice from these results is that there seems to be little difference between the results where $\vartheta_{jj}$ is assumed known and the results where $\vartheta_{jj}$   is estimated from the data.


\medskip

In a typical use of the method in the applications, one would run the MCMC sampler only once, and compute the posterior estimate, and confidence intervals, for instance  as in Section \ref{sec:estimation}. We show one such output. In Setting (a), where $p=100$, we plot on Figure \ref{fig:1} all the $95\%$ confidence intervals for all the off-diagonal elements $\vartheta_{ij}$ of $\vartheta$, obtained from one MCMC run. We also add a dot to the confidence interval line to represent the true value of $\vartheta_{ij}$. The results seem consistent with the results in Table \ref{table:1}.


\medskip

\begin{table}[h]
\begin{center}
\small
\scalebox{.9}{\begin{tabular}{cccc}
\hline
 & $\vartheta^2_{jj}$ known & Empirical Bayes &  \texttt{Glasso} \\
\hline
Relative Error ($\mathcal{E}$ in $\%$) & 19.2 & 21.6 & 63.1\\
Sensitivity ($\textsf{SEN}$ in $\%$) & 68.4 & 69.0 & 40.5\\
Precision ($\textsf{PREC}$ in $\%$) &  100.0  & 100.0  & 74.9\\
\hline
\end{tabular}}
\caption{\small{Table showing the relative error, sensitivity and precision (as defined in (\ref{sen:prec})) for Setting (a), with $p=100$ nodes. Based on 20 simulation replications. Each MCMC run is $5\times 10^4$ iterations.
}}\label{table:1}
\end{center}
\end{table}

\begin{table}[h]
\begin{center}
\small
\scalebox{.9}{\begin{tabular}{cccc}
\hline
 & $\vartheta^2_{jj}$ known & Empirical Bayes &  \texttt{Glasso} \\
\hline
Relative Error ($\mathcal{E}$ in $\%$) & 23.1 & 26.2 & 45.2\\
Sensitivity ($\textsf{SEN}$ in $\%$) & 44.6 & 45.4 & 87.9\\
Precision ($\textsf{PREC}$ in $\%$) &  100  & 99.9  & 56.1\\
\hline
\end{tabular}}
\caption{\small{Table showing the relative error, sensitivity and precision (as defined in (\ref{sen:prec})) for Setting (b), with $p=500$ nodes. Based on 20 simulation replications. Each MCMC run is $5\times 10^4$ iterations.
}}\label{table:2}
\end{center}
\end{table}

\begin{table}[h]
\begin{center}
\small
\scalebox{.9}{\begin{tabular}{cccc}
\hline
 & $\vartheta^2_{jj}$ known & Empirical Bayes &  \texttt{Glasso} \\
\hline
Relative Error ($\mathcal{E}$ in $\%$) & 30.8 & 35.2 & 66.9\\
Sensitivity ($\textsf{SEN}$ in $\%$)   & 16.3 & 16.4 & 6.6 \\
Precision ($\textsf{PREC}$ in $\%$)    & 99.9 & 99.8 & 94.7\\
\hline
\end{tabular}}
\caption{\small{Table showing the relative error, sensitivity and precision (as defined in (\ref{sen:prec})) for Setting (c), with $p=1,000$ nodes. Based on 20 simulation replications. Each MCMC run is $5\times 10^4$ iterations.
}}\label{table:3}
\end{center}
\end{table}

\medskip

\begin{figure}[h!]
\centering
\scalebox{0.5}[.6]{\begin{tabular}{c}
\includegraphics{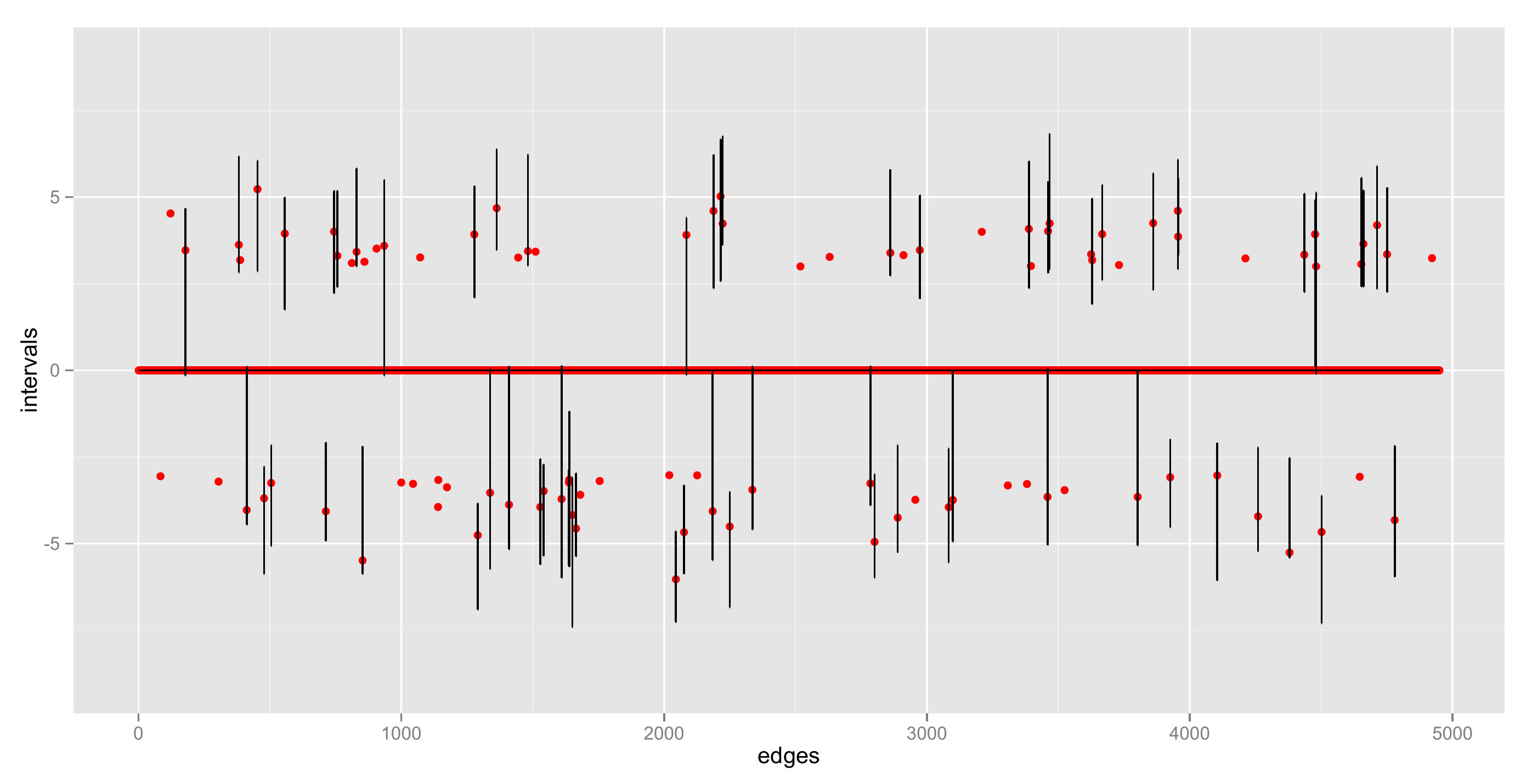} \\
\includegraphics{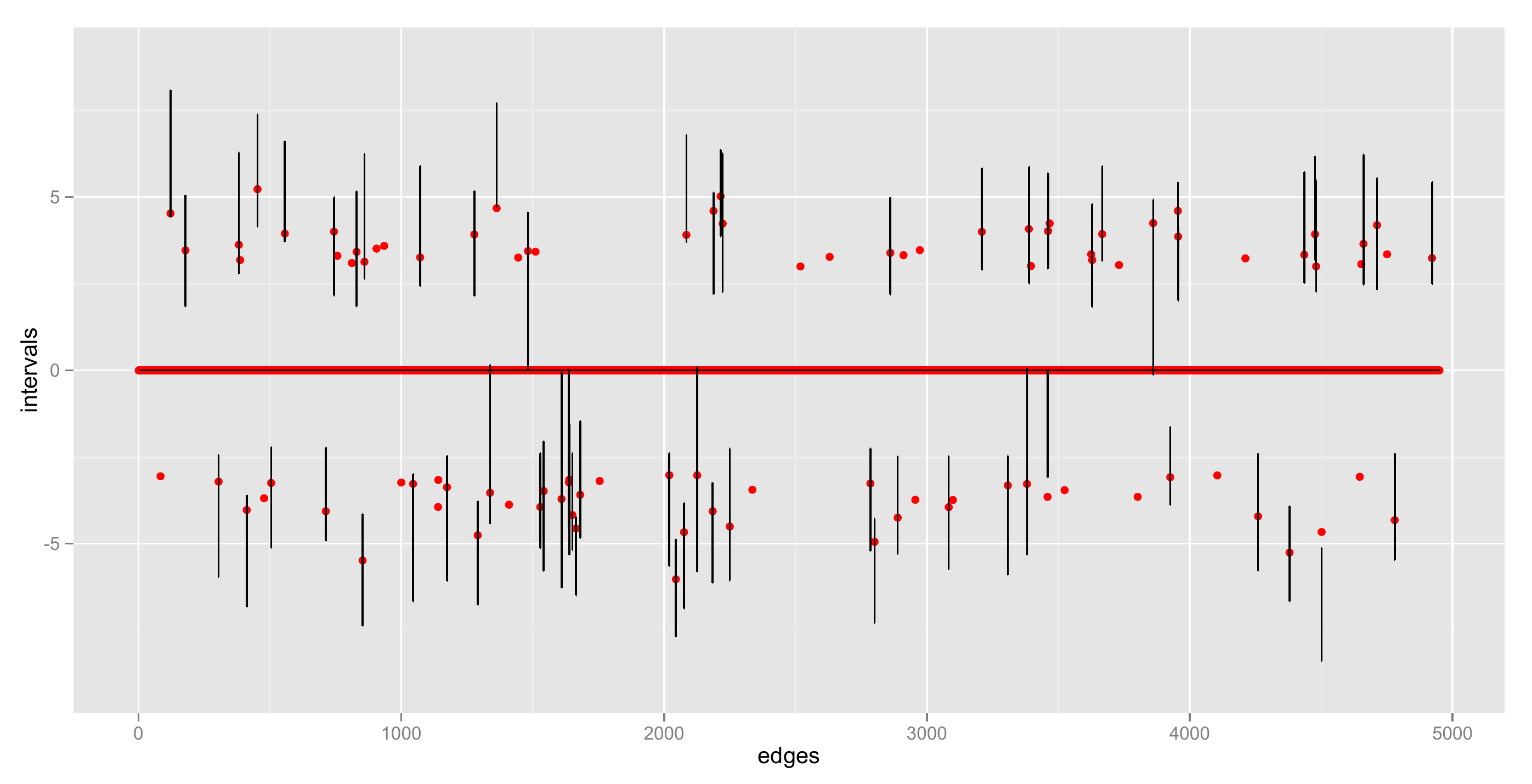} \\
\end{tabular}}
\caption{{\small{Figure showing the confidence interval bars (obtained from one MCMC run), for the non-diagonal entries of $\vartheta$ in Setting (a). The dots represent the true values. Top graph is the case where the $\vartheta_{jj}$'s are assumed known.} }}
\label{fig:1}
\end{figure}

\section{Some concluding remarks}\label{sec:conclusion}
We have developed in this work a quasi-Bayesian methodology for inferring high-dimensional Gaussian graphical models by neighborhood selection. We have shown by examples that, using a high-performance computer systems with multiple cores, the method can fit Gaussian graphical models at a scale unmatched by existing Bayesian methodologies. The general discussion in Section \ref{sec:qbgm} also shows that the method can be easily extended to handle other classes of graphical models. We have studied the asymptotic behavior of the method in the Gaussian case, and showed that for sparse and well-behaved problems, the quasi-posterior distribution concentrates around the true value even in setting where $p$ exceeds $n$. One important direction for future work is the extension of the methodology to estimate the scale parameters $\vartheta_{jj}$ jointly, as part of the Bayesian modeling. This extension raises several difficulties, in terms of the computations (the approximation scheme of \cite{atchade:15a} cannot readily handle such cases), but also in terms of the Bayesian asymptotics.

\section{Proof of Theorem \ref{thm1} and Theorem \ref{thm2}}\label{sec:proof}
Similar results have been derived recently for the linear regression model by \cite{castillo:etal:14}, and by the author in \cite{atchade:15b}. Therefore,  a natural strategy to proof Theorem \ref{thm1} and Theorem \ref{thm2} is to reduce the problem to a corresponding problem in a linear regression model. In the details, we will rely on the behavior of some restricted and $m$-sparse eigenvalues concepts that we introduce first. For $z\in\rset^{n\times q}$, for some $q\geq 1$, and for $s\geq 1$,  we define
\[
\underline{\kappa}(s,z) \eqdef \inf_{\delta\in\{0,1\}^q:\;\|\delta\|_0\leq s}\inf \left\{\frac{\theta'(z'z)\theta}{n\|\theta\|_2^2}:\;\theta\in\rset^q,\;\theta\neq 0,\;\sum_{k:\; \delta_{k}=0}|\theta_k|\leq 7 \sum_{k:\; \delta_{k}=1}|\theta|\right\},\]
and
\begin{multline*}\utilde{\kappa}(s,z)  \eqdef \inf \left\{\frac{\theta'(z'z)\theta}{n\|\theta\|_2^2}:\;\theta\in\rset^q,\;1\leq \|\theta\|_0\leq s\right\},\;\\
\;\tilde{\kappa}(s,z)  \eqdef \sup \left\{\frac{\theta'(z'z)\theta}{n\|\theta\|_2^2}:\;\theta\in\rset^q,\;1\leq \|\theta\|_0\leq s\right\}.
\end{multline*}
In the above definition, we convene that $\inf\emptyset=+\infty$, and $\sup\emptyset=0$. We are interested in the behavior of $\underline{\kappa}(s_\star,X)$, $\utilde{\kappa}(s,X)$ and $\tilde\kappa(s,X)$, when $X$ is the random matrix obtained from assumption H\ref{A1}. We will use the following result taken from \cite{raskutti:etal:10}~Theorem 1, and \cite{rudelson:zhou:13}~Theorem 3.2, which relates the behavior of $\underline{\kappa}(s_\star,X)$, $\utilde{\kappa}(s,X)$ and $\tilde\kappa(s,X)$ to the corresponding term $\underline{\kappa}$, $\utilde{\kappa}(s)$ and $\tilde\kappa(s)$ of the true precision matrix $\K$ introduced in (\ref{u_kappa:1})-(\ref{u_kappa:2}).

\begin{lemma}\label{lem:lem0}
Assume H\ref{A1}. Then there exists finite universal constant $a_1>0$, $a_2>0$ such that for the following hold.
\begin{enumerate}
\item If $\underline{\kappa}>0$, then for all $n\geq a_1\frac{\tilde\kappa(1)}{\underline{\kappa}}s_\star\log(p)$
\[\PP\left[64\underline{\kappa}(s_\star,X)<\underline{\kappa}\right]\leq e^{-a_2 n}.\]
\item Let $1\leq s\leq p$ be such that $\utilde{\kappa}(s)>0$, then for all $n\geq a_1 s\log(p)$,
\[\PP\left[4\utilde{\kappa}(s,X)<\utilde{\kappa}(s)\;\mbox{ or }\; 4\tilde\kappa(s,X)>9\tilde\kappa(s)\right]\leq  e^{-a_2 n}.\]
\end{enumerate}
\end{lemma}

\subsection{Proof of Theorem \ref{thm1}}\label{sec:proof:thm1}
We have
\[\check\Pi_{n,p}(\rmd \theta\vert X)  =  \prod_{j=1}^p \check\Pi_{n,p,j}(\rmd \theta_{\cdot j}\vert X),\]
where for $j\in\{1,\ldots,p\}$, and using that $\alpha=1$ in (\ref{rep:EN}), $\check\Pi_{n,p,j}(\rmd \theta_{\cdot j}\vert X)$ is given by
\begin{equation}\label{qp:proof}
\check\Pi_{n,p,j}(\rmd u\vert X) \propto q_{j}(u;X) \sum_{\delta\in\Delta_p} \pi_\delta \left(\frac{\rho_j}{2\sigma_j^2}\right)^{\|\delta\|_1} e^{-\frac{\rho_j}{\sigma_j^2}\|u\|_1}\mu_{\delta}(\rmd u),\end{equation}
and
\[\log q_{j}(u;X) = -\frac{1}{2\sigma_j^2}\|X_{\cdot j} - X^{(j)}u\|_2^2.\]

For $q\geq 1$, we define  
\[\mathcal{G}_{n,q}\eqdef \left\{ z\in\rset^{n\times q}:\; \tilde\kappa(s_\star,z)\leq \frac{9}{4}\tilde\kappa(s_\star),\;\;\tilde\kappa(1,z)\leq \frac{9}{4}\tilde\kappa(1),\;\mbox{ and }\; \underline{\kappa}(s_\star,z)\geq \frac{\underline{\kappa}}{64}\right\}.\]
For any $k_j\geq 0$, we start by noting that 
\begin{multline*}
\PE\left[\check\Pi_{n,p}\left(\left\{\theta\in\rset^{(p-1)\times p}:\;\|\theta_{\cdot j}\|_0\geq k_j,\;\mbox{ for some }j\right\}\vert X\right)\right]\\
\leq  \PP(X\notin \mathcal{G}_{n,p})  + \sum_{j=1}^p \PE\left[\textbf{1}_{\mathcal{G}_{n,p}}(X) \check\Pi_{n,p,j}\left(\A_j\vert X\right)\right].\end{multline*}
where $\A_j\eqdef \{u\in\rset^{p-1}:\;\|u\|_0 \geq k_j\}$. We notice that if $X\in \mathcal{G}_{n,p}$, then $X^{(j)}\in \mathcal{G}_{n,p-1}$ for any $1\leq j\leq p$. We recall that the notation $X^{(j)}$ denotes the matrix obtained by removing the $j$ column of $X$. Hence
\begin{multline*}
\PE\left[\textbf{1}_{\mathcal{G}_{n,p}}(X) \check\Pi_{n,p,j}\left(\A_j\vert X\right)\right] \\
\leq  \PE\left[\textbf{1}_{\mathcal{G}_{n,p-1}}(X^{(j)}) \check\Pi_{n,p,j}\left(\A_j\vert X\right)\right] = \PE\left[\textbf{1}_{\mathcal{G}_{n,p-1}}(X^{(j)})\PE\left(\check\Pi_{n,p,j}\left(\A_j\vert X\right) \vert X^{(j)}\right)\right].\end{multline*}
We conclude that
\begin{multline}\label{eq1:proof:thm1}
\PE\left[\check\Pi_{n,p}\left(\left\{\theta\in\rset^{(p-1)\times p}:\;\|\theta_{\cdot j}\|_0\geq k_j,\;\mbox{ for some }j\right\}\vert X\right)\right]\\
\leq  \PP(X\notin \mathcal{G}_{n,p})  + \sum_{j=1}^p \PE\left[\textbf{1}_{\mathcal{G}_{n,p-1}}(X^{(j)}) T_j\right],\end{multline}
where
\[T_j = \PE\left(\check\Pi_{n,p,j}\left(\A_j\vert X\right) \vert X^{(j)}\right).\]
The main idea of the proof is to notice that $T_j$ is an expected quasi-posterior probability in the linear regression model $X_{\cdot j} = X^{(j)}\beta +\eta$, where $\eta\sim \textbf{N}(0,(1/\vartheta_{jj})I_n)$. Therefore, by a similar argument and similar calculations as in the proof of Theorem 13 of \cite{atchade:15b}, we have
\begin{multline}\label{eq2:proof:thm1}
T_j \leq 2p\exp\left(-\frac{\vartheta_{jj}\rho_j^2}{8\max_{k\neq j}\|X_{\cdot k}\|_2^2}\right) \\
+ 
2 (4^{s_{\star j}})\left(1+\frac{\sigma_j^2L_j}{\rho_j^2}\right)^{s_{\star j}} e^{\frac{2\rho_j^2 s_{\star j}}{\tau_j\sigma_j^2}} {p-1\choose s_{\star j}}\left(\frac{4c_2}{p^{c_4}}\right)^{k_j-s_{\star j}},
\end{multline}
where $L_j = n\tilde\kappa(s_\star,X^{(j)})$, and $\tau_j = n\underline{\kappa}(s_\star,X^{(j)})$. Since $\max_{k\neq j}\|X_{\cdot k}\|_2^2 = n\tilde\kappa(1,X^{(j)})$, for $X^{(j)}\in\mathcal{G}_{n,p-1}$, it is easy to see that the first term on the right-hand side of (\ref{eq2:proof:thm1}) is bounded by 
\[2p\exp\left(-\frac{\vartheta_{jj}\rho_j^2}{18n\tilde\kappa(1)}\right) = \frac{2}{p^2},\]
where the equality follows from the choice of $\rho_j$. Using the fact that for $X^{(j)}\in\mathcal{G}_{n,p-1}$, we have $L_j \leq (9/4)n\tilde\kappa(s_\star)$, $\tau_j \geq (1/64)n\underline{\kappa}$, it is easy to show that the second term on the right-hand side of (\ref{eq2:proof:thm1}) is bounded by 
\[2\exp\left[s_{\star j}\log(p)\left(\frac{6912}{\sigma_j^2\vartheta_{jj}}\frac{\tilde\kappa(s_\star)}{\underline{\kappa}} +\frac{\sigma_j^2\vartheta_{jj}}{24(\log(p))^2}\frac{\tilde\kappa(s_{\star j})}{\tilde\kappa(1)} +\frac{\log(4ep)}{\log(p)}\right) -\frac{c_4}{2}(k_j-s_{\star j})\log(p)\right].\]
With $k_j=\zeta_j$ as given in the statement of the theorem,  this latter expression is bounded by $2/(p^2)$. This conclude the proof.

\medskip
\subsection{Proof of Theorem \ref{thm2}}\label{sec:proof:thm2}
We use the same approach as above. We define $\bar s_j = s_{\star j} + \zeta_j$ ($\bar s_j=1$ if $s_{\star j}=0$), and  $\bar s = \max_j \bar s_j$, and we set
\[\mathcal{G}_{n,q}\eqdef \left\{ z\in\rset^{n\times q}:\; \tilde\kappa(s_\star,z)\leq \frac{9}{4}\tilde\kappa(s_\star),\mbox{ and }\; \utilde{\kappa}(\bar s,z)\geq \frac{1}{4}\utilde{\kappa}(\bar s)\right\}.\]
We also define $\mathcal{U}\eqdef \{\theta\in\rset^{(p-1)\times p}:\; \|\theta_{\cdot j}-\theta_{\star \cdot j}\|_2> \epsilon_j,\;\mbox{ for some } j\}$, 
$\bar{\mathcal{U}} \eqdef \mathcal{U} \cap \{\theta\in\rset^{(p-1)\times p}:\; \|\theta_{\cdot j}-\theta_{\star \cdot j}\|_0\leq s_{\star j} + \zeta_j \;\mbox{ for all } j\}$, and 
\begin{multline}\label{eq1:proof:thm2}
\check\Pi_{n,p}(\mathcal{U}\vert X) \leq \check\Pi_{n,p}\left(\{\theta\in\rset^{(p-1)\times p}:\; \|\theta_{\cdot j}-\theta_{\star \cdot j}\|_0 > s_{\star j} + \zeta_j\;\mbox{ for some } j\}\vert X\right) \\
+ \textbf{1}_{\mathcal{G}_{n,p}^c}(X) + \textbf{1}_{\mathcal{G}_{n,p}}(X) \check\Pi_{n,p}\left(\bar{\mathcal{U}}\vert X\right).\end{multline}
If for some $j$, $\|\theta_{\cdot j}-\theta_{\star \cdot j}\|_0 > s_{\star j} + \zeta_j$, then we necessarily have $\|\theta_{\cdot j}\|_0>\zeta_j$. Therefore, by Theorem \ref{thm1}, we have:
\begin{equation}\label{proof:thm2:bound1}
\PE\left[\check\Pi_{n,p}\left(\{\theta\in\rset^{(p-1)\times p}:\; \|\theta_{\cdot j}-\theta_{\star \cdot j}\|_0 > s_{\star j} + \zeta_j\;\mbox{ for some } j\}\vert X\right)\right]\leq \frac{2}{e^{a_2n}} + \frac{4}{p}.\end{equation}
By Lemma \ref{lem:lem0}, for $n\geq a_1\bar s\log(p)$,
\begin{equation}\label{proof:thm2:bound2}
\PE\left[\textbf{1}_{\mathcal{G}_{n,p}^c}(X)\right]= \PP\left[X\notin\mathcal{G}_{n,p}\right]  \leq \frac{1}{e^{a_2n}}.\end{equation}
It remains to control the last term on the right-hand side of (\ref{eq1:proof:thm2}). To do so, we note that if $X\in \mathcal{G}_{n,p}$, then $X^{(j)}\in\mathcal{G}_{n,p-1}$ for all $1\leq j\leq p$. Hence
\begin{eqnarray}
\PE\left[\textbf{1}_{\mathcal{G}_{n,p}}(X) \check\Pi_{n,p}\left(\bar{\mathcal{U}}\vert X\right)\right] & \leq & \sum_{j=1}^p \PE\left[\textbf{1}_{\mathcal{G}_{n,p-1}}(X^{(j)})\check\Pi_{n,p,j}(\A_j\vert X)\right]\nonumber\\
&\leq &\sum_{j=1}^p \PE\left[\textbf{1}_{\mathcal{G}_{n,p-1}}(X^{(j)})\PE\left(\check\Pi_{n,p,j}(\A_j\vert X)\vert X^{(j)}\right)\right],
\end{eqnarray}
where $\A_j\eqdef \{u\in\rset^{p-1}:\; \|u-\theta_{\star \cdot j}\|_2>\epsilon_j,\;\mbox{ and }\; \|u-\theta_{\star \cdot j}\|_0\leq \bar s_{j}\}$. 
As in the proof of Theorem \ref{thm1}, we note that under the conditional distribution of $X_{\cdot j}$ given $X^{(j)}$, the term $\check\Pi_{n,p,j}(\A_j\vert X)$ can be viewed as the posterior distribution in the linear regression model $X_{\cdot j} = X^{(j)}\beta +\eta$, where $\eta\sim \textbf{N}(0,(1/\vartheta_{jj})I_n)$. Therefore, by proceeding as in the proof of Theorem 13 of \cite{atchade:15b}, and for any constant $M_0\geq 96$, we have
\begin{multline}\label{eq2:proof:thm2}
\PE\left(\check\Pi_{n,p,j}(\A_j\vert X)\vert X^{(j)}\right) \leq 2p\exp\left(-\frac{\vartheta_{jj}\rho_j^2}{8\max_{k\neq j}\|X_{\cdot k}\|_2^2}\right)\\
 + e^{\bar s_j\log(24pe)} \frac{e^{-\frac{M_0^2\tau_j\bar\epsilon_j^2}{32}} }{1-e^{-\frac{M_0^2\tau_j \bar\epsilon_j^2}{32}}}
+2{p\choose s_{\star j}}\left(\frac{p^{c_3}}{c_1}\right)^{s_{\star j}} \left(1+\frac{\sigma_j^2 L_j}{\rho_j^2}\right)^{s_{\star j}}\frac{e^{-\frac{M_0^2\tau_j\bar\epsilon_j^2}{64}} }{1-e^{-\frac{M_0^2\tau_j\bar\epsilon_j^2}{64}}},
\end{multline}
where $\bar\epsilon_j = \frac{\rho_j \bar s_j^{1/2}}{\tau_j}$, $\tau_j = n\utilde{\kappa}(\bar s_j,X^{(j)})$, and $L_j = n\tilde{\kappa}(s_{\star j},X^{(j)})$. As seen in the proof of Theorem \ref{thm1}, the first term on the right-hand side of (\ref{eq2:proof:thm2}) is upper bounded by $2/p^2$.

We have 
\[\frac{M_0^2\tau_j\bar\epsilon_j^2}{32} \geq \left(\frac{54 M_0^2}{32}\frac{1}{\sigma_j^2\vartheta_{jj}}\right) \bar s_j\log(p).\]
 Hence for $p\geq 24e$, and $\frac{54 M_0^2}{32}\frac{1}{\sigma_j^2\vartheta_{jj}}\geq 4$, the second term on the right-hand side of (\ref{eq2:proof:thm2}) is also upper bounded by $2/p^2$. For $\frac{54 M_0^2}{32}\frac{1}{\sigma_j^2\vartheta_{jj}}\geq 4$, the third term is upper bounded by
\begin{multline*}
4\exp\left[s_{\star j}\log(p)\left(2+c_3 +\frac{\sigma_j^2\vartheta_{jj}}{24(\log(p)^2)} \frac{\tilde\kappa(s_{\star j})}{\tilde\kappa(1)}\right) - \frac{54 M_0^2}{64}\frac{1}{\sigma_j^2\vartheta_{jj}}\bar s_j\log(p)\right]\leq \frac{2}{p^2},
\end{multline*}
 by choosing $\frac{54 M_0^2}{64}\frac{1}{\sigma_j^2\vartheta_{jj}}\geq 2 + \frac{c_4}{2}(2+c_3)$. This conclude the proof.

\vspace{1cm}
\medskip
\medskip

\noindent {\large{\textbf{Acknowledgements}}}

The author would like to thank Shuheng Zhou for very helpful conversations. This work is partially supported by the NSF, grants DMS 1228164 and DMS 1513040.

\bibliographystyle{ims}
\bibliography{biblio_graph,biblio_mcmc,biblio_optim}

\end{document}